\documentclass[11pt,english,a4paper]{amsart}

\usepackage{amsfonts,amssymb,amsmath,amsthm}
\usepackage[english,french]{babel}
\usepackage{courier,newlfont}

\usepackage[utf8x]{inputenc} 

\usepackage[dvips]{graphicx}
\usepackage{color,graphics, wrapfig, subfig}

\usepackage{tabls}

\usepackage[numbers]{natbib}

\usepackage{shorttoc}

\usepackage{eurosym}

\usepackage[pagebackref=true, hyperfootnotes=false]{hyperref}

\usepackage[symbol,stable]{footmisc}
\usepackage{perpage}
\MakePerPage[1]{footnote}

\usepackage{url}
\usepackage{multirow, multicol}

\usepackage{todonotes}

\newtheorem{theorem}{Theorem}[section]
\newtheorem{proposition}[theorem]{Proposition}

\theoremstyle{definition}

\newtheorem{example}{Example}
\theoremstyle{remark}
\newtheorem*{remark}{Remark}

\numberwithin{equation}{section}

\newcommand{\R}{\ensuremath{\mathbb{R}}}

\newcommand{\E}{\ensuremath{\mathbb{E}}}

\newcommand{\var}{\mbox{Var}}
\newcommand{\cov}{\mbox{Cov}}

\newcounter{noteTabCap} 

\title{Estimation of quantile oriented sensitivity indices}
	
\author{V\'eronique Maume-Deschamps}
\address{Universit\'e de Lyon, Universit\'e Lyon 1, Institut Camille Jordan ICJ UMR 5208 CNRS}
\email{veronique.maume@univ-lyon1.fr}
\author{Ibrahima Niang}
\address{Altran Research}
\email{ibrahima.niang@altran.com}

\begin{document}
\selectlanguage{english}

\maketitle


\begin{abstract}
The paper concerns quantile oriented sensitivity analysis. We rewrite the corresponding indices using the Conditional Tail Expectation risk measure. Then, we use this new expression to built estimators.
\end{abstract}



\section{Introduction}
Many models encountered in applied sciences involve input parameters which are often not precisely known. Some of the input variables may strongly affect the output, while others have a small effect. Sensitivity analysis aims at measuring the impact of each input parameter uncertainty on the model output and, more specifically, to identify the most sensitive parameters (or groups of parameters). One of the common metrics to evaluate the sensitivity is the Sobol index. 
More precisely, given two random variables $X$ and $Y$, the $X$-Sobol index on $Y$  compares the total variance of $Y$ to the expected variance of the variable $Y$ conditioned by $X$, i.e:
\begin{equation}
\label{Sobol_indivi}
S_X= \frac{\var(\E( Y\ |\ X))}{\var(Y)}.
\end{equation}
It rewrites $S_X=1- \frac{\displaystyle\E((\var Y \ |\ X))}{\var(Y)}$. It is a statistical measure of the relative impact of $X$ on the variability of $Y$; the most sensitive parameters can be identified and ranked as the parameters with the largest Sobol indices.
Introduced in \cite{fort2013new}, contrast indices named "Goal Oriented Sensitivity Analysis" generalize the Sobol ones. The construction of these indices is based on contrast functions.
Roughly speaking, given  $\psi$ a contrast function and $X, Y$ two random variables, the $\psi$-$X$ contrast index on $Y$ is given by (see  \cite{fort2013new}):
\begin{equation*}
S_{X}^{\psi}=\frac{\underset{\theta \in \mathbb{R} }{\operatorname{min}}\left(\E[ \psi(Y;\theta)]\right)-\E(\underset{\theta \in \mathbb{R} }{\operatorname{min}}\left[\E(\psi(Y;\theta)|X)\right])}{\underset{\theta \in \mathbb{R} }{\operatorname{min}}\left(\E [\psi(Y;\theta)]\right)-\E\left(\underset{\theta \in \mathbb{R} }{\operatorname{min}}\psi(Y;\theta)\right)}.
\end{equation*}
$S_{X}^{\psi}$ represents a statistical indicator of the impact of $X$ on the variability of the output function $Y$ with respect to the argmin of a contrast function. If we consider the mean-constrat function $\psi: (y,\theta) \mapsto (y-\theta)^{2}$, then we retrieve the first order Sobol indices defined in (\ref{Sobol_indivi}) and it measures the impact of $X$ to the deviation of $Y$ with respect to the mean. 
There exists a pretty large literature in the case where the contrast function is given by the mean-constrat function $\psi: (y,\theta) \mapsto (y-\theta)^{2}$ (see for example \cite{sobol2001global}, \cite{borgonovo2006measuring}, \cite{borgonovo2007quantification}, \cite{saltelli2008global} or \cite{janon2012analyse}).\\
In this  paper, we focus on contrast indices obtained with the $\alpha$-quantile contrast function, i.e., with the contrast function given by:
\begin{equation}
\label{quan_index_1}
\psi_{\alpha}: (y,\theta) \mapsto (y-\theta)(\alpha-\mathbf{1}_{y\leq \theta}), \ \ \alpha \in ]0,1[
\end{equation}
We call quantile contrast index the sensitivity measure that one obtain with the contrast function given by (\ref{quan_index_1}). This setting is also called {\em quantile oriented sensitivity analysis (QOSA)} (see \cite{ioss}). In this work, we show that quantile contrast indices can be linked with the Conditional Tail Expectation (or CTE) risk measure. 
Then, from this expression, we propose an estimation method based on Monte Carlo sampling techniques. Note that this approach is the first formal estimation of the quantiles oriented sensitivity indices. It appears firstly in \cite{ibrahima_these}. Nevertheless, during the writing of this article, we have been informed of a recent work (\cite{ioss}) where another procedure is proposed. 
\\
\ \\
The paper is organized as follows. In Section \ref{sec:def}, we recall basic definitions on the CTE and the contrast indices. In Section \ref{sec:link_cte}, we show that quantile contrast may be rewritten in terms of the CTE. Finally, Section \ref{sec:asym_norm}, is devoted to the estimation of quantile contrast indices and to some simulations.

\section{Main definitions}\label{sec:def}
Let us consider a measurable function $f$ and a random vector $X=(X_1,\ldots,X_k) \in \mathbb{R}^{k}$, $k \geq 1$. Let $Y$ be the real-valued response variable: $Y=f(X)$.
We assume that the random input variables $X_i, i \geq 1$ are independent, which is the standard setting of sensitivity analysis.\\
For a random variable $Z$, $F_Z$ denotes its distribution function, and $F_{Z}^{-1}$ the generalized inverse of $F_Z$ (or the quantile function). We recall that for a continuous random variable $Z$ with moment of order $1$, the Conditional Tail Expectation denoted CTE of $Z$ at level $\alpha \in ]0, 1[$ is given by:
\begin{equation}
CTE_{\alpha}(Z)=\E\left[Z|Z>F^{-1}_{Z}(\alpha)\right]=\frac{1}{1-\alpha}\int_{\alpha}^{1}F_{Z}^{-1}(u)du.
\end{equation}
From an actuarial point of view, the  CTE measures the average of losses given that a specified confidence level $\alpha$ is exceed. We refer the reader to \cite{denuit2006actuarial} for a detailed review.
In what follows, we introduce some concepts about contrast functions and contrast index. The reader is referred for instance to \cite{fort2013new}. We assume that the output $Y=f(X)$ is a continuous random variable. 
Given a function $\psi$
$$\begin{array}{ccccc}
\psi & : & \mathbb{R} \times \mathbb{R} & \to & \mathbb{R}^+ \\
 & & (y, \theta) & \mapsto & \psi(y,\theta),\\
\end{array}$$
such that there is a unique minimum  $\theta^{*} \in \mathbb{R}$ to $ \E[ \psi(Y;\theta)]$.
Some examples of contrast functions that allow to estimate various parameters associated to a probability distribution are listed on Table \ref{Table:Levy}. 
\begin{table}[h]	
\begin{center}
{\renewcommand{\arraystretch}{1.8}
\begin{tabular}{|l|l|l|}
\cline{2-3} \multicolumn{1}{c|}{}  & Contrast functions  & $\theta^{*}$ \\
		\hline
		 The mean &  $\psi(y,\theta)=|y-\theta|^2$  &  $\theta^{*}=\E(Y)$  \\

		\hline
		 The median &  $\psi(y,\theta)=|y-\theta|$ & $\theta^{*}=F_{Y}^{-1}(\frac{1}{2})$   \\
		 \hline
		The $\alpha$-quantile &  $\psi(y,\theta)=(y-\theta)(\alpha-\mathbf{1}_{y\leq \theta})$ \ & $\theta^{*}=F_{Y}^{-1}(\alpha)$   \\
		 \hline
		\end{tabular}
		}
\end{center}
\caption{Contrast functions exemples}
\label{Table:Levy}	
\end{table}

For a more exhaustive list of contrast functions, the reader is referred for instance to \cite{fort2013new}.\\ 
In what follows, we focus on $\alpha$- quantile contrast function, i.e the contrast function given by:
\begin{equation}
\label{cont_eq}
\psi_{\alpha}: (y,\theta) \mapsto (y-\theta)(\alpha-\mathbf{1}_{y\leq \theta}), \ \ \alpha \in ]0, 1[
\end{equation}
The quantile oriented index associated to the input $X_i$ is:
\begin{equation}\label{eq:qosa}
S_{X_i}^\alpha = \frac{\displaystyle \min_{\theta\in\R} (\E[\psi_\alpha(Y\/;\theta)])-\E(\min_{\theta\in\R} \E[\psi_\alpha(Y\/;\theta)|X_i])}{\displaystyle \min_{\theta\in\R} (\E[\psi_\alpha(Y\/;\theta)])}
\end{equation}
\begin{remark}
It is straightforward to see that if $Y$ and $X_i$ are independent then $S_{X_i}^\alpha=0$ and if $Y$ is $X_i$-measurable then $S_{X_i}^\alpha=1$, which is an expected behavior for sensitivity indices.
\end{remark}
We are interested on how one can express the quantile contrast index in terms of the Conditional Tail Expectation (CTE) risk measure.
\section{Relation between contrast indices and Conditional Tail Expectation (CTE) risk measure}\label{sec:link_cte}
We recall that, for $X$ and $Y$ two random variables, the conditional cumulative distribution $F^{}_{Y|X}$ and the conditional quantile $F^{-1}_{Y|X}$ at level $\alpha$ are defined respectively as:
$$F^{}_{Y|X}(y)=\mathbb{P}\left( Y\leq y|X\right), \ \ \ y \in \mathbb{R}.$$
\begin{equation}
F^{-1}_{Y|X}(\alpha)=\inf \left\{y \in \mathbb{R}|\ F^{}_{Y|X}(y)\geq \alpha \right\}.
\end{equation}
These are $X$-measurable random functions.
In the following proposition, we give a relation between the risk measure CTE and quantile contrast index.
\begin{proposition} Assume that  $\E|Y|<\infty$,  then $\forall \alpha \in ]0, 1[$, $\forall i=1\/,\ldots\/,k$,
\begin{equation}
S_{X_i}^{{\alpha}}=1-\frac{\E\left(Y|Y>F^{-1}_{Y|X_i}(\alpha)\right)-\E(Y)}{CTE_{\alpha}(Y)-\E(Y)}\/.
\end{equation}
\end{proposition}
\begin{proof}
$S_{X_i}^{{\alpha}}$ rewrites as:
\begin{equation}
S_{X_i}^{{\alpha}}=1-\frac{\E(\underset{\theta \in \mathbb{R} }{\operatorname{min}}\left[\E(\psi_{\alpha}(Y;\theta)|X_i)\right])}{\underset{\theta \in \mathbb{R} }{\operatorname{min}}\left(\E[ \psi_{\alpha}(Y;\theta)]\right)}.
\end{equation}
Consider $\theta_{1}^{*}= \underset{\theta \in \mathbb{R}}{\operatorname{argmin}} \E[ \psi_{\alpha}(Y;\theta)]$. We have $\theta_1^*=F^{-1}(\alpha)$.

On one other side, 
$$\E(\underset{\theta \in \mathbb{R} }{\operatorname{min}}\left[\E(\psi_{\alpha}(Y;\theta)|X_i)\right])=\E\left(\E\left(\psi_{\alpha}(Y;\theta_{2}^{*})|X_i\right)\right)=\E\left(\psi_{\alpha}(Y;\theta_{2}^{*})\right),$$
whit $\theta_{2}^{*}= \underset{\theta \in \mathbb{R}}{\operatorname{argmin}} \E\left[ \psi_{\alpha}(Y;\theta)|X_i\right]$ it is a $X_i$ measurable random variable. Let  $Z=Y-F_{Y|X_i}^{-1}(\alpha)$.\\
Remark that, $\frac{\partial \E\left[ \psi_{\alpha}(Y;\theta)|X_i\right]}{\partial \theta}=-\alpha + \mathbb{P}(Y\leq \theta|X_i)$ and $\theta_{2}^{*}=F_{Y|X_i}^{-1}(\alpha)$. We deduce:
\begin{eqnarray*}
\E\left(\psi_{\alpha}(Y;\theta_{2}^{*})\right)&=& \E\left(Z1_{Y>F_{Y|X_i}^{-1}(\alpha)}\right) -(1-\alpha)\E(Z)\\
&=&\mathbb{P}(Y>\theta_{2}^{*})\E(Y|Y>\theta_{2}^{*}) -\mathbb{P}(Y>\theta_{2}^{*})\E(\theta_{2}^{*}|Y>\theta_{2}^{*})-(1-\alpha)\E(Z).
\end{eqnarray*}
In addition, 
$$\mathbb{P}(Y>\theta_{2}^{*})=1-\mathbb{P}(Y\leq\theta_{2}^{*})=1-\E\left(\E(1_{Y\leq \theta_{2}^{*}}|X_i)\right)=1-\E\left(F_{Y|X_i}(\theta_{2}^{*})\right)=1-\alpha\/.$$
On one other side, conditioning by $X_i$, leads to
$$\E(\theta_{2}^{*}|Y>\theta_{2}^{*})=\frac{1}{1-\alpha}\E\left(\theta_{2}^{*}1_{Y>\theta_{2}^{*}}\right)=\frac{1}{1-\alpha}\E(\E\left(\theta_{2}^{*}1_{Y>\theta_{2}^{*}}|X_i\right))=\E(\theta_{2}^{*}).$$
Consequently,
\begin{eqnarray*}
\E\left(\psi_{\alpha}(Y;\theta_{2}^{*})\right)&=& (1-\alpha)\E(Y|Y>\theta_{2}^{*})-(1-\alpha)\E(\theta_{2}^{*})-(1-\alpha)\E(Y) +(1-\alpha)\E(\theta_{2}^{*})\\
&=&(1-\alpha)\left(\E(Y|Y>\theta_{2}^{*})-\E(Y) \right).
\end{eqnarray*}
which concludes the proof.
\end{proof}
We shall use this expression of the quantile oriented indices to propose an estimation of them. 
\section{Estimation of quantile contrast index}\label{sec:asym_norm}
We consider the statistical estimation of quantile contrast index. Fix an index $i=1\/,\ldots\/,k$.
We assume that $\E|Y|<\infty$. We consider, for $\alpha \in ]0, 1[$. We have 
\begin{equation}
\label{New_formula_anal}
S_{X_i}^{\alpha}=1-\frac{\E\left(Y|Y>F^{-1}_{Y|X_i}(\alpha)\right)-\E(Y)}{CTE_{\alpha}(Y)-\E(Y)}. 
\end{equation}
 Except for some very specific and simple examples, an analytic formula for $S_{X_i}^\alpha$ is cannot be reached. 
 Hence, it is natural to wonder how these indices could be estimated. In the case where the contrast function is given by mean-contrast functions which correspond to Sobol index, there exist a pretty large literature dedicated to the estimation of such index (see, e.g., \cite{sobol2001global,saltelli2002making,tarantola2006random,janon2012analyse,janon2014asymptotic}). In this section, we first describe how quantile contrast index can be estimated.
 We propose an estimation Monte-Carlo  replication based. Indeed, we use two independent $n$-samples $(X_1^j\/,\ldots\/, X_k^j)$, $(X_1^{*j}\/,\ldots\/, X_k^{*j})$, $j=1\/,\ldots\/, n$. 
\subsection{Estimation procedure}
Let $\theta^{*}=F^{-1}_{Y}(\alpha)$ and $\theta_{i}(x)=F^{-1}_{Y|X_{i}=x}(\alpha)$ be the respective quantile at level $\alpha$ of $Y$ and of $Y|X_{i}=x$. \\
Consider $Y_1, \ldots, Y_n$ are $n, n \geq 1$ an i.i.d. $n$-sample of the distribution of $Y$. 
$F_n$ denotes the empirical distribution function: 
$$F_n(y) = \displaystyle\frac{1}{n}\sum_{i=1}^{n}1_{\left\{Y_i \leq y\right\}}, \forall y \in \mathbb{R}$$
 and $\widehat{\theta}^*$ is the classical quantile estimator. \\
The conditional distribution function $F_{Y | X_i=x}$ can be estimated by a Kernel estimator:
$${F}_{n}(y|X_i=x)=\frac{\sum_{j=1}^{n}K \left(\frac{x-X_{i}^{j}}{h_n} \right)1_{\left\{Y_j \leq y\right\}}}{\sum_{j=1}^{n}K \left(\frac{x-X_{i}^{j}}{h_n} \right)},$$
where $K$ is a kernel and $h_n$ is a positive number depending on the sample size $n$, called bandwidth. Then $\theta_i(x)$ is obtained by
\begin{equation}
\label{quantil_cond_estimator}
\hat{\theta}_{i}(x)={F}_{n}^{-1}(\alpha|x)=\inf \left\{y:{F}_{n}(y|x) \geq \alpha \right\},
\end{equation}
We refer the reader to \cite{bhattacharya1990kernel}; \cite{koenker1996conditional}; \cite{gannoun2003nonparametric} or \cite{takeuchi2006nonparametric} for a detailed review on quantile and conditional quantile estimation.\\
\ \\

  In what follows, we give an estimation procedure for estimating quantile contrast index. It  requires the two following steps (recall that $Y=f(X_1,\ldots, X_k)$)
  \begin{enumerate}
\item Generate $X_{1}^j,\ldots,X_{k}^j$ and compute the $Y_j=f(X_{1}^j,\ldots,X_{k}^j)$, for $j=1,\ldots, n$. Replace in (\ref{New_formula_anal}) the expectation $\E(Y)$ by its empirical version.
\item  Generate $X^{*j}=(X_{1}^{*j},\ldots,X_{k}^{*j})$ for $j=1,\ldots, n$ (independent copies of the previous vectors) and compute the $Y^{*}_j=f(X_{1}^{*j},\ldots,X_{k}^{*j})$, $j=1,\ldots, n$. Then from the sample $Y^{*}_j=f(X_{1}^{*j},\ldots,X_{k}^{*j})$, $j=1,\ldots, n$ compute $\hat{\theta}^{*}$ and $ \hat{\theta}_{i}(X_{i})$.
\end{enumerate}
%
An estimator of the index (\ref{New_formula_anal}) is given by
 \begin{equation}
 \label{est_empiriq}
S_{n\/,X_i}^{\alpha}=1-\frac{\frac{1}{n(1-\alpha)}\sum_{j=1}^{n}Y_{j}^{} 1_{Y_{j}^{}>\hat{\theta}_{i}(X_{i}^j)} - \frac{1}{n}\sum_{j=1}^{n}Y_{j}^{}}{\frac{1}{n(1-\alpha)}\sum_{j=1}^{n}Y_j 1_{Y_{j}>\hat{\theta}^*} - \frac{1}{n}\sum_{j=1}^{n}Y_j}.
\end{equation}
Remark that equation (\ref{est_empiriq}) is equivalent to: 
\begin{equation}
\label{est_empi_suite}
S_{n\/,X_i}^{\alpha}=1-\frac{\frac{1}{n}\sum_{j=1}^{n}{R_{j}^{i}}}{\frac{1}{n}\sum_{j=1}^{n}{Z_j}},
\end{equation}
where, for $j=1, \ldots, n$:
$$R_{j}^{i}=Y_{j}^{} \left(\frac{1}{1-\alpha}\textbf{1}_{Y_{j}^{}>\hat{\theta}_{i}(X_{i}^{j})} -1 \right), \ \ \  Z_j=Y_j \left(\frac{1}{1-\alpha}\textbf{1}_{Y_{j}>\hat{\theta}^*} -1 \right).$$
conditionally to the sample $X^{*}$, are two independent and identically distributed (i.i.d) sample of the distribution of $R^{i}$ and $Z$ where 
$$R^{i}=Y\left(\frac{1}{1-\alpha}\mathbf{1}_{Y>F^{-1}_{Y|X_i}(\alpha)}-1 \right), \ \ \ Z=Y\left(\frac{1}{1-\alpha}\mathbf{1}_{Y>F^{-1}_{Y}(\alpha)}-1 \right).$$
For the sake of simplicity, we denote by $$\bar{Z}_n=\frac{1}{n}\sum_{j=1}^{n}{Z_j},\ \ \ \  \bar{R}_n=\frac{1}{n}\sum_{j=1}^{n}{R_{j}^{i}}.$$
\begin{proposition} (Consistency). Assume that  $\E|Y|<\infty$, then 
\begin{equation}
S_{n\/,X_i}^{\alpha} \overset{a.s}{\underset{n\to\infty}{\longrightarrow}}S_{X_i}^{{\alpha}}
\end{equation}

\end{proposition}
\proof
The result is a straightforward application of the strong law of large number conditionally to $X^*$. 
\endproof

\begin{proposition} (Asymptotic normality). Assume that $Y$ is square integrable. Then,
\label{conv1}
\begin{equation}
 \sqrt{n} (S_{n\/,X_i}^{\alpha} - S_{X_i}^{{\alpha}}) \overset{\mathcal{L}}{\underset{n\to\infty}{\longrightarrow}}\mathcal{N}\left(0, \sigma_S^2\right),
 \end{equation}
 where $$\sigma_S^2=\frac{\var(R^{i}) -2\beta(1-S_{X_i}^{{\alpha}}) + (1-S_{X_i}^{{\alpha}})^2 \var(Z)}{\left(CTE_{\alpha}(Y)-\E(Y)\right)^2},$$
 and $\beta=\cov(R^{i},Z)$.
\end{proposition}
\proof
 We do the proof conditionally to $X^*$. Denote
 $$S_{n\/,X_i}^{\alpha}=h(\bar{U}_n),$$
 where $U_j=(R_{j}^{i}, Z_j)^T$ and $$h(x,y)=1-\frac{x}{y}.$$
 The central limit theorem gives that:
$$ \sqrt{n}(\bar{U}_n - \mu) \overset{\mathcal{L}}{\underset{n\to\infty}{\longrightarrow}}\mathcal{N}_{2}\left(0,\Gamma \right),$$ 
where $\Gamma$ is the covariance matrice of $(R^{i},Z)$ and $\mu=\left(\begin{array}{c} \E(R^{i}) \\ \E(Z) \end{array} \right)$. Since the function $h$ is differentiable at $\mu$, the Delta method gives:
$$ \sqrt{n} (S_{n\/,X_i}^{ \alpha} - S_{X_i}^{{\alpha}}) \overset{\mathcal{L}}{\underset{n\to\infty}{\longrightarrow}}\mathcal{N}\left(0, g^{T}\Gamma g\right),$$
where $$g=\nabla h(\mu).$$
By differentiation, we get that, for any $x,y$ so that $y \neq 0$:
$$\nabla h(x,y)=\left(-\frac{1}{y}, \frac{x}{y^2} \right)^T.$$
Hence, using \ref{New_formula_anal}, we get that $$g=\left(\frac{S_{X_i}^{{\alpha}} -1}{\E(R^{i})},\frac{(1-S_{X_i}^{{\alpha}})^2}{\E(R^{i})} \right)^T$$
Thus
\begin{equation*}
\begin{aligned}
g^{T}\Gamma g &=& \var(R^{i})\frac{(1-S_{X_i}^{{\alpha}})^2}{\E(R^{i})^2}-2\cov(R^{i},Z)\frac{(1-S_{X_i}^{{\alpha}})^3}{\E(R^{i})^2}+\var(Z)\frac{(1-S_{X_i}^{{\alpha}})^4}{\E(R^{i})^2}\\
&=&\frac{(1-S_{X_i}^{{\alpha}})^2}{\E(R^{i})^2}\left(\var(R^{i})-2\cov(R^{i},Z)(1-S_{X_i}^{{\alpha}}) + \var(Z)(1-S_{X_i}^{{\alpha}})^2\right).
\end{aligned}
\end{equation*}
Remark that $\frac{(1-S_{X_i}^{{\alpha}})^2}{\E(R^{i})^2}=\frac{1}{\E(Z)^2}$ and that $\E(Z)=CTE_{\alpha}(Y)-\E(Y)$, hence we deduce that
$$g^{T}\Gamma g=\frac{\var(R^{i}) -2\beta(1-S_{X_i}^{\psi_{\alpha}}) + (1-S_{X_i}^{{\alpha}})^2 \var(Z) }{\left(CTE_{\alpha}(Y)-\E(Y)\right)^2},$$
where $\beta=\cov(R^{i},Z)$ which concludes the proof.
\endproof
\begin{remark}
When computing $\sigma_S^2$, we replace the variance of $R^{i}$ and $Z$ as well as the covariance of the random vector $(R^{i}, Z)$ by their empirical versions.
\end{remark}
\begin{remark}
Following \cite{janon2014asymptotic}, in order to improve the efficiency of the estimator, $CTE_\alpha(Y)$ and $\E(Y)$ could be estimated by using the complete sample $X,X^*$. A further study is needed to get its asymptotic properties. 
\end{remark}
\subsection{Numerical illustrations}
In order to validate our estimation procedure, we illustrate the asymptotic results of Proposition \ref{conv1} in the following example that is considered in \cite{fort2013new}. 
\begin{example}
Let us considerer an output of type
$$Y=f(X_1,X_2)=X_1 + X_2.$$
where $X_1 \sim Exp(1)$ and $X_2 \sim -X_1$ with $X_1$ and $X_2$ independent. The quantile contrast index of $X_1$ and $X_2$ are known analytically (see \cite{fort2013new}) and they are given by
\begin{equation*}
S_{X_1}^{\alpha}=
  \left\{
    \begin{aligned}
     \frac{(1-\alpha)(1-\log(2(1-\alpha)))+\alpha \log(\alpha)}{(1-\alpha)(1-\log(2(1-\alpha)))} \ \ \ if \ \ \  \alpha \geq \frac{1}{2}\\
       \frac{\alpha(1-\log(2\alpha))+\alpha \log(\alpha)}{\alpha(1-\log(2\alpha))} \ \ \ if \ \ \  \alpha < \frac{1}{2}\\
    \end{aligned}
  \right.
\end{equation*}
and
\begin{equation*}
S_{X_2}^{\alpha}=
  \left\{
    \begin{aligned}
     \frac{(1-\alpha)(1-\log(2(1-\alpha)))+(1-\alpha) \log(1-\alpha)}{(1-\alpha)(1-\log(2(1-\alpha)))} \ \ \ if \ \ \  \alpha \geq \frac{1}{2}\\
       \frac{\alpha(1-\log(2\alpha))+(1-\alpha) \log(1-\alpha)}{\alpha(1-\log(2\alpha))} \ \ \ if \ \ \  \alpha < \frac{1}{2}\\
    \end{aligned}
  \right.
\end{equation*}
In the case where $\alpha=\frac{1}{2}$, then $S_{X_1}^{\alpha}=S_{X_2}^{\alpha}$.\\ 
Table \ref{tab_vasi1} displays the sensitivity of $X_1$ and $X_2$ for different values of $\alpha$. We estimate for each variable a $95\%$ confidence interval.  We denote by $S^{X_1}_{n, \alpha}$ and $S^{X_2}_{n, \alpha}$ as the empirical estimator of $S_{X_1}^{\alpha}$ and $S_{X_2}^{\alpha}$ given in  (\ref{est_empiriq}) for a sample size $n=100000$. The conditional quantile of $Y|X_1$ and $Y|X_2$ are computed using a gaussian kernel with a bandwidth $h_n=n^{-\frac{1}{5}}$. The quantity $IC_1$ and $IC_2$ denote the respective confidence interval of $X_1$ and $X_2$. Table \ref{tab_vasi1} presents, for different values of $\alpha$, the relative mean square error (RMSE) of $X_1$ and $X_2$ which measures the relative average of the square of the "errors", that is the following quantity
$$RMSE_{X_i}=\sqrt{\frac{1}{n}\sum_{j=1}^{n}{\left( \frac{(S^{X_i}_{n, \alpha})_j - S_{X_i}^{\alpha}}{S_{X_i}^{\alpha}}\right)^{2}}}, \ \ i=1,2.$$
As we can see in Table \ref{err_quad}, the RMSE is low which assess the reliability of our estimator. From this result, we can conclude that the estimated indices $S^{X_i}_{n, \alpha}, i=1,2$ is near to the true values modulo the Monte-
Carlo errors due to the numerical simulation of the indices.
\begin{table}[h]
\centering
\begin{tabular}{|c|c|c|c||c|c|c|}
\hline 
$\alpha$&$S_{X_1}^{\alpha}$&$S_{n\/,X_1}^{\alpha}$& $IC_1$&$S_{X_2}^{\alpha}$&$S_{n\/,X_2}^{\alpha}$&$IC_2$\\
\hline
$0.05$&$0.0929$&$0.0950$& $[0.0807,0.1052]$ &$0.7049$&$0.7005$ &$[06917,0.7181]$\\
\hline
$0.1$&$0.1176$&$0.1166$&$[0.1097,0.1255]$&$0.6366$&$0.6397$&$[0.6259,0.6474]$\\
\hline
$ {0.5}$&$ {0.3069}$&$0.3108$& $[0.3009,0.3128]$&$ {0.3069}$&$0.3042$&$[0.3010,0.3127]$\\
\hline
$0.7$&$0.4491$&$0.4436$&$[0.4417,0.4566]$&$0.2031$&$0.2026$&$[0.1980,0.2082]$\\
\hline
$0.99$&$0.7974$&$0.7907$&$[0.7756,0.8193]$&$0.0625$&$0.0630$&$[0.0309,0.0941]$\\
\hline
\end{tabular}
\caption{Quantile contrast index confidence interval at level $95\%$ for different values of $\alpha$.}
\label{tab_vasi1}
\end{table}
\begin{table}[h]
\centering
\begin{tabular}{|c|c|c||c|c|c|c|}
\hline 
$\alpha$&$S_{X_1}^{\alpha}$&$RMSE$&$S_{X_2}^{\alpha}$&$RMSE$\\
\hline
$0.05$&$0.0929$&$0.0318$&$0.7049$&$0.0064$\\
\hline
$0.1$&$0.1176$&$0.0241$&$0.6366$&$0.0065$\\
\hline
$ {0.5}$&$ {0.3069}$&$0.0091$&$ {0.3069}$&$0.0096$\\
\hline
$0.7$&$0.4491$&$0.0075$&$0.2031$&$0.0130$\\
\hline
\hline
\end{tabular}
\caption{Relative mean square error (RMSE).}
\label{err_quad}
\end{table}
\end{example}
\begin{example}(Vasicek Model)
Here we present a financial application of  the use of quantile contrast index. We focus on the classical Vasicek model where the yield curve is given as an output of an instantaneous spot rate model with the following risk-neutral dynamics
 \begin{equation}\label{e:vasicek}
dr_t=a(b-r_t)dt + \sigma dW_{ t}
\end{equation}
where $a$,\ $b$ and $\sigma$ are positive constants and where $W$ is a standard brownian motion. Parameter $\sigma$ is the volatility of the short rate process, $b$ corresponds to the long-term mean-reversion level whereas $a$ is the speed of convergence of the short rate process $r$ towards level $b$. The price at time $t$ of a zero coupon bond  with maturity $T$ in such a model is given by (see, e.g., \cite{brigo2007interest}):
\begin{equation}
\label{Survival_Probe}
P(t,T)=A(t,T)e^{-r_tB(t,T)}
\end{equation}
where 
$$A(t,T)=\exp \left( (b -\frac{\sigma^2}{2a^2} )(B(t,T) -(T-t)) -\frac{\sigma^2}{4a}B^{2}(t,T) \right)$$
and
$$B(t,T)=\frac{1-e^{-a(T-t)}}{a}.$$
The yield-curve can be obtained as a deterministic transformation of zero-coupon bond prices at different maturities.\\
In what follows, we quantify the sensitivity of the input parameters  $\left \{a, b, \sigma \right \}$ affecting the uncertainty in the bond price at time $t=0$. In the following numerical experiments, the maturity $T$ and the initial spot rate $r_0$ are chosen such that $T=1$ and $r_0=10\%$. Table \ref{tab_vasi1_alpha_quan}  reports the estimated quantile contrast indices of the parameter $a, b, \sigma$  for different values of $\alpha$  where  the probability laws of these parameters are uniform on $[0, 1]$. For this model, no closed form formula is available.
\begin{table}[h]
\centering
\begin{tabular}{|c|c|c||c|c|c|c|}
\hline 
$\alpha$&$\widehat{S}_{a}^{\alpha}$&$\widehat{S}_{b}^{\alpha}$&$\widehat{S}_{\sigma}^{\alpha}$\\
\hline
$0.05$&$0.4961$&$\textcolor{red}{0.5961}$&$0.0210$\\
\hline
$0.1$&$0.1948$&$0.4036$&$0.0667$\\
\hline
$0.5$&$0.1722$&$0.2685$&$0.0508$\\
\hline
$0.7$&$0.1096$&$0.1679$&$\textcolor{red}{0.1913}$\\
\hline
$0.9$&$0.1053$&$0.0025$&$0.2928$\\
\hline
$0.99$&$\textcolor{red}{0.4744}$&$0.2596$&$0.3965$\\
\hline
\end{tabular}
\caption{Estimation of $\alpha$-quantile contrast indices for different values of $\alpha$}\label{tab_vasi1_alpha_quan}
\end{table}
Contrary to Sobol indices, the sensitivity of our model parameters strongly depend on the confidence level $\alpha$. For $\alpha \in \left \{0.05, 0.1, 0.5, \right \}$, most of the uncertainty in the bond price is due to the long-term mean-reversion level $b$ and the speed of mean reversion $a$ whereas for $\alpha=0.9$ or $\alpha=0.99$, we can see that $a$ and $\sigma$ becomes more important than the the long-term mean-reversion level $b$ in terms of output uncertainty.
\end{example}
\section{Conclusion}\label{sec:con}
The main goal of this paper was to propose an estimation of the quantile oriented sensitivity indices. Our proposition if based on a rewriting of these indices using the Conditional Tail Expectation. This domain deserve much additional work in order to make the QOSA a useful and practical tool. In particular, their performance with respect to the procedure proposed in \cite{ioss} has to be studied. 


\bibliographystyle{plain}  
\bibliography{biblio_rachdi}






%
%

\end{document}